\documentclass[a4paper,11pt]{article}
\usepackage{amsmath}
\usepackage{amssymb}
\usepackage{theorem}
\usepackage{pstricks}
\usepackage{euscript}
\usepackage{epic,eepic}
\usepackage{graphicx}
\PassOptionsToPackage{normalem}{ulem}
\newcommand{\menge}[2]{\big\{{#1} \mid {#2}\big\}}

\newcommand{\emp}{\ensuremath{{\varnothing}}}

\newcommand{\scal}[2]{\left\langle{#1}\mid {#2} \right\rangle}

\newcommand{\vuo}{\ensuremath{\mbox{\footnotesize$\square$}}}

\newcommand{\HH}{\ensuremath{\mathcal H}}
\newcommand{\XX}{\ensuremath{\mathcal X}}
\newcommand{\GG}{\ensuremath{\mathcal G}}
\newcommand{\YY}{\ensuremath{\mathcal Y}}

\newcommand{\HHH}{\ensuremath{\boldsymbol{\mathcal H}}}
\newcommand{\XXX}{\ensuremath{\boldsymbol{\mathcal X}}}
\newcommand{\YYY}{\ensuremath{\boldsymbol{\mathcal Y}}}

\newcommand{\KKK}{\ensuremath{\boldsymbol{\mathcal K}}}
\newcommand{\GGG}{\ensuremath{\boldsymbol{\mathcal G}}}

\newcommand{\CCC}{\ensuremath{\boldsymbol{C}}}
\newcommand{\MM}{\ensuremath{\boldsymbol{M}}}
\newcommand{\NNN}{\ensuremath{\boldsymbol{N}}}

\newcommand{\sri}{\ensuremath{\operatorname{sri}}}

\newcommand{\RR}{\ensuremath{\mathbb R}}

\newcommand{\NN}{\ensuremath{\mathbb N}}

\newcommand{\dom}{\ensuremath{\operatorname{dom}}}

\newcommand{\prox}{\ensuremath{\operatorname{prox}}}

\newcommand{\cart}{\ensuremath{\mbox{\huge{$\times$}}}}

\newcommand{\argmin}{\ensuremath{\operatorname{argmin}}}
\newcommand{\ran}{\ensuremath{\operatorname{ran}}}
\newcommand{\zer}{\ensuremath{\operatorname{zer}}}
\newcommand{\gra}{\ensuremath{\operatorname{gra}}}

\newcommand{\vv}{\ensuremath{\boldsymbol{v}}}
\newcommand{\sss}{\ensuremath{\boldsymbol{s}}}
\newcommand{\xx}{\ensuremath{\boldsymbol{x}}}
\newcommand{\pp}{\ensuremath{\boldsymbol{p}}}
\newcommand{\qq}{\ensuremath{\boldsymbol{q}}}
\newcommand{\yy}{\ensuremath{\boldsymbol{y}}}

\newcommand{\rr}{\ensuremath{\boldsymbol{r}}}

\newcommand{\zz}{\ensuremath{\boldsymbol{z}}}
\newcommand{\bb}{\ensuremath{\boldsymbol{b}}}
\newcommand{\cc}{\ensuremath{\boldsymbol{c}}}

\newcommand{\aaa}{\ensuremath{\boldsymbol{a}}}

\newcommand{\BB}{\ensuremath{\boldsymbol{B}}}
\newcommand{\LL}{\ensuremath{\boldsymbol{L}}}

\newcommand{\AAA}{\ensuremath{\boldsymbol{A}}}
\newcommand{\BBB}{\ensuremath{\boldsymbol{B}}}

\newcommand{\DD}{\ensuremath{\boldsymbol{D}}}

\newcommand{\Id}{\ensuremath{\operatorname{Id}}}

\newcommand{\weakly}{\ensuremath{\rightharpoonup}}

\newtheorem{theorem}{Theorem}[section]

\newtheorem{corollary}[theorem]{Corollary}

\theoremstyle{plain}{\theorembodyfont{\rmfamily}
}
\theoremstyle{plain}{\theorembodyfont{\rmfamily}
}
\theoremstyle{plain}{\theorembodyfont{\rmfamily}
}
\theoremstyle{plain}{\theorembodyfont{\rmfamily}
\newtheorem{example}[theorem]{Example}}
\theoremstyle{plain}{\theorembodyfont{\rmfamily}
\newtheorem{problem}[theorem]{Problem}}
\theoremstyle{plain}{\theorembodyfont{\rmfamily}
\newtheorem{remark}[theorem]{Remark}}
\theoremstyle{plain}{\theorembodyfont{\rmfamily}
}

\definecolor{labelkey}{rgb}{0,0.08,0.45}
\definecolor{refkey}{rgb}{0,0.6,0.0}
\definecolor{Brown}{rgb}{0.45,0.0,0.05}
\definecolor{dgreen}{rgb}{0.00,0.49,0.00}
\definecolor{dblue}{rgb}{0,0.08,0.75}
\RequirePackage[dvips,colorlinks,hyperindex]{hyperref}
\hypersetup{linktocpage=true,citecolor=dblue,linkcolor=dgreen}
\numberwithin{equation}{section}

\begin{document}
\title{\sffamily\huge 
A splitting algorithm for system of composite monotone inclusions}
\author{ D\hspace{-1.6ex}\raise 0.4ex\hbox{-}\hspace{.8ex}inh D\~ung$^1$
and B$\grave{\text{\u{a}}}$ng C\^ong V\~u$^2$\\[5mm]
\small$^1$ Information Technology institut\\
\small 144 Xuan Thuy, Cau Giay, Ha Noi Vietnam,
 \url{dinhzung@gmail.com}\\
\small$^2$  Vietnam national university Hanoi\\
\small Department of mathematics,\\
\small  334 Nguyen Trai, Thanh Xuan, Hanoi, Vietnam,
\small\url{vu@ann.jussieu.fr}\\
}

\date{}
\maketitle
\begin{abstract}
We propose a splitting algorithm for 
solving a system of composite monotone inclusions formulated in 
the form of the extended set of solutions in 
real Hilbert spaces. The resluting algorithm 
is a an extension of the algorithm in \cite{plc13b}.
The weak convergence of the algorithm proposed is proved. 
Applications to minimization problems is demonstrated.
\end{abstract}



\section{Introduction}
Let $\HH$ be a real Hilbert space,
let $A\colon\HH\to 2^{\HH}$ be a set-valued operator.
The domain and the graph of $A$ are respectively defined by 
$\dom A=\menge{x\in\HH}{Ax\neq\emp}$ and 
$\gra A=\menge{(x,u) \in \HH\times\HH}{u\in Ax}$.
We denote by $\zer A=\menge{x\in\HH}{0\in Ax}$ the set of zeros 
of $A$, and by 
$\ran A=\menge{u\in\HH}{(\exists\; x\in\HH)\; u\in Ax}$ 
the range of $A$. The inverse of $A$ is
$A^{-1}\colon\HH\mapsto 2^{\HH}\colon u\mapsto 
\menge{x\in\HH}{u\in Ax}$.
Moreover, $A$ is monotone if 
\begin{equation}
 (\forall(x,y)\in\HH\times \HH)\;
 (\forall (u,v)\in Ax\times Ay)\quad \scal{x-y}{u-v} \geq 0,
\end{equation}
and maximally monotone if it is monotone and there exists no 
monotone operator 
$B\colon\HH\to2^\HH$ such that $\gra B$ properly contains $\gra A$. 

\noindent 
A basis problem in monotone operator theory is to find 
a zero point of the sum of two maximally monotone operators $A$
and $B$ acting on a real Hilbert space $\HH$, that is, find 
$\overline{x}\in\HH$ such that
\begin{equation}
\label{prob:1}
 0\in A\overline{x} + B\overline{x}.
\end{equation}
Suppose that the problem \eqref{prob:1} has at least one solution $\overline{x}$. Then 
there exists $\overline{v}\in B\overline{x}$ such that $-\overline{v}\in A\overline{x}$.
The set of all such pairs $(\overline{x},\overline{v})$ define the extended set of solutions 
to the problem \eqref{prob:1} \cite{Eck08},
\begin{equation}
 E(A,B) = \menge{(\overline{x},\overline{v})}{\overline{v}\in B\overline{x},-\overline{v}\in A\overline{x} }.
\end{equation}
Inversely, if $E(A,B)$ is non-empty and $(\overline{x},\overline{v})\in E(A,B)$, 
then the set of solutions to the problem \eqref{prob:1} is also nonempty since $\overline{x}$
solves \eqref{prob:1} and $\overline{v}$ solves its dual problem \cite{Atto96}, i.e,
\begin{equation}
 0\in B^{-1}v -A^{-1}(-v). 
\end{equation}
It is remarkable that three fundamental methods such as Douglas-Rachford splitting method,
forward-backward splitting method, forward-backward-forward splitting method
converge weakly to points in $E(A,B)$ \cite[Theorem 1]{Svaiter11}, \cite{Siop1}, \cite{Tseng00}.
We next consider a more general problem where one of the operator has a linearly composite.
In this case, the problem \eqref{prob:1} becomes \cite[Eq. (1.2)]{siop2},
 \begin{equation}
\label{prob:2}
 0\in A\overline{x} + (L^*\circ B\circ L)\overline{x},
\end{equation}
where $B$ acts on a real Hilbert space $\GG$ and $L$ 
is a bounded linear operator from $\HH$ to $\GG$. Then, it is shown 
in \cite[Proposition 2.8(iii)(iv)]{siop2} that whenever the set of solutions to 
\eqref{prob:2} is non-empty, the extended set of solutions
\begin{equation}
 E(A,B,L) = \menge{(\overline{x},\overline{v})}{-L^*\overline{v}\in A\overline{x}, 
L\overline{x} \in B^{-1}\overline{v}}
\end{equation}
is non-empty and, for every $(\overline{x},\overline{v}) \in E(A,B,L)$, 
$\overline{v}$ is a solution to the dual problem of \eqref{prob:2} \cite[Eq.(1.3)]{siop2},
\begin{equation}
\label{prob:2b}
 0\in B^{-1}\overline{v} - L\circ A^{-1}\circ(-L^*)\overline{v}.
\end{equation}
Algorithm proposed in \cite[Eq.(3.1)]{siop2} to solve the pair \eqref{prob:2} and \eqref{prob:2b} 
converges weakly to a point in $E(A,B,L)$ \cite[Theorem 3.1]{siop2}. Let us  consider 
the case when monotone inclusions involving the parallel-sum monotone operators. 
This typical inclusion is firstly introduced in 
\cite[Problem 1.1]{plc6} and then studied in \cite{Bang12} and \cite{Bot}. A simple case is  
 \begin{equation}
\label{prob:3}
 0\in A\overline{x} + L^*\circ(B\;\vuo\; D)\circ L\overline{x} + C\overline{x},   
 \end{equation}
where $B$, $D$ act on $\GG$ and $C$ acts on $\HH$, and the sign $\vuo$
denotes the parallel sums operations defined by
\begin{equation}
 B\;\vuo\; D = (B^{-1}+D^{-1})^{-1}.
\end{equation}
Then under the assumption that the set of solutions to \eqref{prob:3} is non-empty,
so is its extended set of solutions defined by
\begin{equation}
\label{e:ext}
 E(A,B,C,D,L) = \menge{(\overline{x},\overline{v})}{-L^*\overline{v}\in
 (A+C)\overline{x}, L\overline{x} \in (B^{-1}+D^{-1})\overline{v}}.
\end{equation}
 Furthermore, if there exists $(\overline{x},\overline{v})\in E(A,B,C,D,L)$, then 
$\overline{x}$ solves \eqref{prob:3} and $\overline{v}$ solves its dual problems defined by
\begin{equation}
 0\in B^{-1}\overline{v} -L\circ(A+C)^{-1}\circ(-L^*)\overline{v} + D^{-1}\overline{v}.
\end{equation}
Under suitable conditions on operators,
the algorithms in \cite{plc6}, \cite{Bot} and \cite{Bang12} 
converge weakly to a point in $E(A,B,C,D,L)$.
We also note that even in the more complex situation when $B$ and $D$ in \eqref{prob:3} 
admit linearly composites structures introduced firstly \cite{plc13b} and then in \cite{Botb}, 
in this case  \eqref{prob:3} becomes
 \begin{equation}
\label{prob:4}
 0\in A\overline{x} + L^*\circ\Big((M^*\circ B\circ M)\;\vuo\; 
(N^*\circ D\circ N)\Big)\circ L\overline{x} + C\overline{x},   
 \end{equation}
where $M$ and $N$ are respectively bounded linear operator from 
$\GG$ to real Hilbert spaces $\YY$ and $\XX$, $B$ and $D$ act on $\YY$ and $\XX$, respectively, 
under suitable conditions on operators, simple calculations show that, 
the algorithm proposed in \cite{plc13b} and \cite{Botb} converge weakly to the points in the 
extended set of solutions, 
 \begin{equation}
\label{e:ext1}
 E(A,B,C,D,L,M,N) = \menge{(\overline{x},\overline{v})}{-L^*\overline{v}\in
 (A+C)\overline{x}, L\overline{x} \in ((M^*\circ B\circ M)^{-1}+(N^*\circ D\circ N)^{-1})\overline{v}}.
\end{equation}
Furthermore, for each $(\overline{x},\overline{v}) 
\in E(A,B,C,D,L,M,N)$, then $\overline{v}$ solves the dual problem
of \eqref{prob:4},
\begin{equation}
 0\in (M^*\circ B\circ M)^{-1}\overline{v} -L\circ(A+C)^{-1}\circ(-L^*)\overline{v} 
+ (N^*\circ D\circ N)^{-1}\overline{v}.
\end{equation}
To sum up, above analysis shows that 
each primal problem formulation mentioned has a dual problem which admits an 
explicit formulation and the corresponding algorithm converges weakly to a point in
the extended set of solutions. 
However, there is a class of inclusions 
in which their dual problems are no longer available, 
for instance, when $A$ is univariate and $C$ is multivariate, as in 
\cite[Problem 1.1]{plc2010}. Therefore, 
it is necessary to find a new way to overcome this limit. 
Observer that the problem in the form of \eqref{e:ext1} can recover both the primal problem and 
dual problem. Hence, it will be more convenience to 
formulate the problem in the form of 
\eqref{e:ext1} to overcome this limitation.
This approach is firstly used in \cite{Bang13b}. In this paper we extend 
it to the following problem to unify 
some recent primal-dual frameworks in the literature. 
\begin{problem}
\label{prob}
Let $m,s$ be strictly positive integers.
For every $i\in \{1,\ldots,m\}$, 
let $(\HH_i,\scal{\cdot}{\cdot})$ be a real Hilbert space,
let $z_i\in\HH_i$, let $A_{i}\colon\HH_i\to 2^{\HH_i}$
be maximally monotone, let $C_i\colon \HH_1\times\ldots\times\HH_m\to\HH_i$ be such that 
\begin{alignat}{2}
\label{coco}
 \Big(\exists \nu_0 \in \left[0,+\infty\right[\Big)
\Big(\forall &(x_i)_{1\leq i\leq m}\in \HH_1\times\ldots\times\HH_m\Big) 
 \Big(\forall (y_i)_{1\leq i\leq m}\in \HH_1\times\ldots\times\HH_m\Big) \notag \quad\\
& \begin{cases}
\sum_{i=1}^m \|C_i(x_1,\ldots, x_m) - C_i(y_1,\ldots, y_m) \|^2
\leq\nu_{0}^2 \sum_{i=1}^m \|x_i-y_i \|^2\\
\sum_{i=1}^m \scal{C_i(x_1,\ldots, x_m) - C_i(y_1,\ldots, y_m)}{x_i-y_i} \geq 0.
\end{cases}
\end{alignat}
For every $k\in\{1,\ldots, s\}$, 
let $(\GG_k,\scal{\cdot}{\cdot})$, $(\YY_k,\scal{\cdot}{\cdot})$ and $(\XX_k,\scal{\cdot}{\cdot})$ 
be real Hilbert spaces, 
let $r_k \in \GG_k$, 
let $B_{k}\colon \YY_k\to2^{\YY_k}$ 
be maximally monotone, let $D_k\colon \XX_k\to 2^{\XX_k}$ be maximally monotone,
let $M_k\colon \GG_k\to \YY_k$ and 
$N_k\colon \GG_k\to \XX_k$  be  bounded linear operators, and 
every $i\in\{1,\ldots, m\}$, 
let $L_{k,i}\colon\HH_i \to\GG_k$ 
be a  bounded linear operator.
The  problem is to find $\overline{x}_1 \in \HH_1,\ldots, \overline{x}_m \in \HH_m$ and 
$\overline{v}_1 \in \GG_1,\ldots, \overline{v}_s \in \GG_s$ such that
\begin{alignat}{2}\label{dualin}
\begin{cases}
z_1-\displaystyle\sum_{k=1}^s L_{k,1}^*\overline{v}_k\in A_1\overline{x}_1 
+ C_1(\overline{x}_1,\ldots,\overline{x}_m) \\
  \vdots\quad\\
z_m-\displaystyle\sum_{k=1}^s L_{k,m}^*\overline{v}_k\in A_m\overline{x}_m
 + C_m(\overline{x}_1,\ldots,\overline{x}_m)\\
\displaystyle\sum_{i=1}^mL_{1,i}\overline{x}_i-r_1 \in 
 ( M^{*}_1\circ B_{1}\circ M_1)^{-1}\overline{v}_1 
+  ( N^{*}_1\circ D_{1}\circ N_1)^{-1}\overline{v}_1\\
  \vdots\quad\\
\displaystyle\sum_{i=1}^mL_{s,i}\overline{x}_i-r_s \in 
( M^{*}_s\circ B_{s}\circ M_s)^{-1} \overline{v}_s +
 ( N^{*}_s\circ D_{s}\circ N_s)^{-1} \overline{v}_s.
\end{cases}
\end{alignat}
We denote by $\Omega$ the set of solutions to \eqref{dualin}. 
\end{problem}
Here are some connections to existing primal-dual problems in the literature.
\begin{enumerate}
 \item In Problem \ref{prob}, 
set $m=1$, $(\forall k\in\{1,\ldots,s\})\; L_{k,1} = \Id$,
then by removing $\overline{v}_1,\ldots, \overline{v}_s$ from \eqref{dualin}, 
we obtain the primal inclusion in \cite[Eq.(1.7)]{plc13b}.
Furthermore, by removing $\overline{x}_1$ from \eqref{dualin}, we obtain the dual inclusion which 
is weaker than the dual inclusion in \cite[Eq.(1.8)]{plc13b}. 
 \item In Problem \ref{prob}, 
set $m=1$, $C_1$ is restricted to be cocoercive (i.e., $C_{1}^{-1}$ is strongly monotone),
then by removing $\overline{v}_1,\ldots, \overline{v}_s$ from \eqref{dualin}, 
we obtain the primal inclusion in \cite[Eq.(1.1)]{Botb}.
Furthermore, by removing $\overline{x}_1$ from \eqref{dualin}, we obtain the dual inclusion which 
is weaker than the dual inclusion in \cite[Eq.(1.2)]{Botb}. 
\item In Problem \ref{prob}, set $(\forall k\in\{1,\ldots,s\})\;   \YY_k = \XX_k =\GG_k$ and 
$M_k= N_k = \Id$, then we obtain the system of inclusions in \cite[Eq.(1.3)]{Bang13b}. Furthermore,
if for every $i\in\{1,\ldots,m\},\; C_i$ is restricted on $\HH_i$ and $(D_{k}^{-1})_{1\leq k\leq s}$ are 
single-valued and Lipschitzian,
then by removing respectively $\overline{v}_1,\ldots, \overline{v}_s$ and
 $\overline{x}_1,\ldots, \overline{x}_m$, we obtain respectively 
the primal inclusion in \cite[Eq.(1.2)]{Combettes13} and the dual inclusion in \cite[Eq.(1.3)]{Combettes13}.
\item In Problem \ref{prob}, set $s=m$, $(\forall i\in\{1,\ldots,m\})\; z_i= 0, A_i = 0$ and 
$(\forall k\in\{1,\ldots,s\})\; r_k =0, (k\neq i)\; L_{k,i} = 0$. 
 Then, we obtain the dual inclusion in 
\cite[Eq.(1.2)]{Botc} where $(D^{-1}_k)_{1\leq k\leq s}$ are single-valued and Lipschitzian. 
Moreover, by removing the variables 
 $\overline{v}_1,\ldots,\overline{v}_s$, we obtain the primal inclusion in \cite[Eq.(1.2)]{Botc}.
\end{enumerate}

In the present paper, we develop the splitting technique in \cite{plc13b} which is reused in \cite{Botb},
and base on the convergence result of the algorithm proposed in \cite{Combettes13},
we propose a splitting algorithm for 
solving Problem \ref{prob} and prove its convergence in Section  \ref{Algocon}. 
We provide some  application examples in the last section.

\noindent{\bf Notations.} (See  \cite{livre1}) 
 The scalars product and the  
norms of all Hilbert spaces used in this paper are denoted respectively by 
$\scal{\cdot}{\cdot}$ and $\|\cdot\|$. 
We denote by $\mathcal{B}(\HH,\GG)$
the space of all bounded linear operators from $\HH$ to $\GG$. 
The symbols $\weakly $ and $\to$ denote respectively 
weak and strong convergence.
The resolvent of $A$ is
\begin{equation}
 J_A=(\Id + A)^{-1}, 
\end{equation}
where $\Id$ denotes the identity operator on $\HH$.
We say that $A$ is uniformly monotone 
at $x\in\dom A$ if there exists an 
increasing function $\phi\colon\left[0,+\infty\right[\to 
\left[0,+\infty\right]$ vanishing only at $0$ such that 
\begin{equation}\label{oioi}
\big(\forall u\in Ax\big)\big(\forall (y,v)\in\gra A\big)
\quad\scal{x-y}{u-v}\geq\phi(\|x-y\|).
\end{equation}
 The class of all lower semicontinuous convex functions 
$f\colon\HH\to\left]-\infty,+\infty\right]$ such 
that $\dom f=\menge{x\in\HH}{f(x) < +\infty}\neq\emp$ 
is denoted by $\Gamma_0(\HH)$. Now, let $f\in\Gamma_0(\HH)$.
The conjugate of $f$ is the function $f^*\in\Gamma_0(\HH)$ defined by
$f^*\colon u\mapsto
\sup_{x\in\HH}(\scal{x}{u} - f(x))$, and the subdifferential 
of $f\in\Gamma_0(\HH)$ is the maximally monotone operator 
\begin{equation}
 \partial f\colon\HH\to 2^{\HH}\colon x
\mapsto\menge{u\in\HH}{(\forall y\in\HH)\quad
\scal{y-x}{u} + f(x) \leq f(y)}
\end{equation} 
with inverse given by
\begin{equation}
(\partial f)^{-1}=\partial f^*.
\end{equation}
Moreover,
the proximity operator of $f$ is
\begin{equation}
\prox_f=J_{\partial f} \colon\HH\to\HH\colon x
\mapsto\underset{y\in\HH}{\argmin}\: f(y) + \frac12\|x-y\|^2.
\end{equation}
\section{Algorithm and convergence}
\label{Algocon}
The main result of the paper can be now stated in which we introduce 
our splitting algorithm, prove its convergence and provide the connections 
to existing work.

\begin{theorem}
\label{t:2} In Problem \ref{prob}, suppose that $\Omega\not=\emp$ and that 
\begin{equation}
\label{e:cons2}
 \beta = \nu_0+\sqrt{\sum_{i=1}^m\sum_{k=1}^s\|N_kL_{k,i}\|^2 
+\max_{1\leq k\leq s}(\|N_k\|^2+\|M_k\|^2)} > 0. 
\end{equation}
For every $i\in\{1,\ldots,m\}$,
let $(a^{i}_{1,1,n})_{n\in\NN}$, $(b^{i}_{1,1,n})_{n\in\NN}$, $(c^{i}_{1,1,n})_{n\in\NN}$ 
be absolutely 
summable sequences in $\HH_i$, 
for every $k\in\{1,\ldots,s\}$,
let $(a^{k}_{1,2,n})_{n\in\NN}$, $(c^{k}_{1,2,n})_{n\in\NN}$ be  
absolutely summable sequences in $\GG_k$,
let $(a^{k}_{2,1,n})_{n\in\NN}$
$(b^{k}_{2,1,n})_{n\in\NN}$, $(c^{k}_{2,1,n})_{n\in\NN}$ absolutely 
summable sequences in $\XX_k$,
$(a^{k}_{2,2,n})_{n\in\NN}$, $(b^{k}_{2,2,n})_{n\in\NN}$, $(c^{k}_{2,2,n})_{n\in\NN}$ 
be  absolutely summable sequences in $\YY_k$. For every $i\in\{1,\ldots,m\}$ and $k\in\{1,\ldots,s\}$,
let $x^{i}_{1,0} \in \HH_i$, $x_{2,0}^k \in \GG_k$ and 
$v_{1,0}^k \in \XX_k$, $v_{2,0}^k \in \YY_k$,  
let $\varepsilon \in \left]0,1/(\beta+1)\right[$, 
let $(\gamma_n)_{n\in\NN}$ be sequence in 
$\left[\varepsilon, (1-\varepsilon)/\beta \right]$ and set
\begin{equation}
\label{e:forwardback}
\begin{array}{l}
\operatorname{For}\;n=0,1,\ldots,\\
\left\lfloor
\begin{array}{l}
\operatorname{For}\;i=1,\ldots, m\\
\left\lfloor
\begin{array}{l}
s_{1,1,n}^i = x_{1,n}^i -\gamma_n\big(C_i(x_{1,n}^1,\ldots,x_{1,n}^m) 
+ \sum_{k=1}^sL_{k,i}^*N_{k}^*v_{1,n}^k + a_{1,1,n}^i\big)\\
p_{1,1,n}^i = J_{\gamma_n A_i}(s_{1,1,n}^i +\gamma_n z_i) + b_{1,1,n}^i \\
\end{array}
\right.\\[2mm]
\operatorname{For}\;k=1,\ldots,s\\
\left\lfloor
\begin{array}{l}
p_{1,2,n}^k = x_{2,n}^k +\gamma_n\big( 
 N_{k}^*v_{1,n}^k - M_{k}^*v_{2,n}^k + a_{1,2,n}^k \big)\\
s_{2,1,n}^k = v_{1,n}^k + \gamma_n\big( 
\sum_{i=1}^mN_kL_{k,i} x_{1,n}^i - N_kx_{2,n}^k + a_{2,1,n}^k \big)\\
p_{2,1,n}^k = s_{2,1,n}^k-\gamma_n \big(N_kr_k+ 
J_{\gamma_{n}^{-1}D_k}(\gamma_{n}^{-1}s_{2,1,n}^k -N_kr_k) + b_{2,1,n}^k \big)\\
q_{2,1,n}^k = p_{2,1,n}^k 
+\gamma_n\big( N_k\sum_{i=1}^mL_{k,i}p_{1,1,n}^i-N_kp_{1,2,n}^k +c_{2,1,n}^k \big) \\
v_{1,n+1}^k = v_{1,n}^k - s_{2,1,n}^k + q_{2,1,n}^k\\
s_{2,2,n}^k = v_{2,n}^k + \gamma_n\big( 
M_kx_{2,n}^k + a_{2,2,n}^k \big)\\
p_{2,2,n}^k = s_{2,2,n}^k-\gamma_n \big( 
J_{\gamma_{n}^{-1}B_k}(\gamma_{n}^{-1}s_{2,2,n}^k) + b_{2,2,n}^k \big)\\
q_{2,2,n}^k = p_{2,2,n}^k 
+\gamma_n\big( M_kp_{1,2,n}^k +c_{2,2,n}^k \big) \\
v_{2,n+1}^k = v_{2,n}^k - s_{2,2,n}^k + q_{2,2,n}^k\\
q_{1,2,n}^k = p_{1,2,n}^k +\gamma_n\big(
N_{k}^*p_{2,1,n}^k - M_{k}^*p_{2,2,n}^k + c_{1,2,n}^k \big)\\
x_{2,n+1}^k = x_{2,n}^k -p_{1,2,n}^k+ q_{1,2,n}^k\\
\end{array}
\right.\\[2mm]
\operatorname{For}\;i=1,\ldots, m\\
\left\lfloor
\begin{array}{l}
q_{1,1,n}^i = p_{1,1,n}^i -\gamma_n\big(C_i(p_{1,1,n}^1,\ldots,p_{1,1,n}^m) 
+ \sum_{k=1}^sL_{k,i}^*N^{*}_kp_{2,1,n}^k + c_{1,1,n}^{i} \big)\\
x_{1,n+1}^i = x_{1,n}^i -s_{1,1,n}^i+ q_{1,1,n}^i.\\
\end{array}
\right.\\[2mm]
\end{array}
\right.\\[2mm]
\end{array}
\end{equation}
Then the following hold for each $i\in\{1,\ldots,m\}$ and $k\in\{1,\ldots,s\}$.
\begin{enumerate}
 \item \label{t:2i}
$\sum_{n\in\NN}\|x_{1,n}^i - p_{1,1,n}^i \|^2 < +\infty $\quad\text{and}\quad 
$\sum_{n\in\NN}\|x_{2,n}^k - p_{1,2,n}^k \|^2 < +\infty $.
 \item \label{t:2ii}
$\sum_{n\in\NN}\|v_{1,n}^k - p_{2,1,n}^k \|^2 < +\infty $\quad \text{and}\quad
$\sum_{n\in\NN}\|v_{2,n}^k - p_{2,2,n}^k \|^2 < +\infty $.
\item \label{t:2iii}
 $x_{1,n}^i\weakly \overline{x}_{1,i}$, 
$v_{1,n}^k\weakly \overline{v}_{1,k}$,
$v_{2,n}^k\weakly \overline{v}_{2,k}$ and 
 such that 
$ M_{k}^*\overline{v}_{2,k} 
=  N_{k}^*\overline{v}_{1,k}$ and that 
$(\overline{x}_{1,1},\ldots,\overline{x}_{1,m},N_{1}^*\overline{v}_{1,1},\ldots,N_{s}^*\overline{v}_{1,s}) \in\Omega$.
\item 
\label{t:3iv} 
Suppose that $A_j$ is uniformly monotone at $\overline{x}_{1,j}$, for some 
$j\in\{1,\ldots,  m\}$, then $x_{1,n}^j \to \overline{x}_{1,j}$. 
\item 
\label{t:3v} 
Suppose that the operator 
$(x_i)_{1\leq i\leq m} \mapsto (C_j(x_i)_{1\leq i\leq m})_{1\leq j\leq m}$ 
is uniformly monotone at 
$(\overline{x}_{1,1},\ldots,\overline{x}_{1,m})$, then $(\forall i\in\{1,\ldots,m\})\; 
x_{1,n}^i\to \overline{x}_{1,i}$.
 \item \label{t:3vi}
Suppose that there exists $j\in\{1,\ldots,m\}$ and an increasing 
function $\phi_j\colon \left[0,+\infty\right[ \to \left[0,+\infty\right]$ vanishing 
only at $0$
such that
\begin{alignat}{2}
\label{csoco}
\Big(\forall (x_i)_{1\leq i\leq m}&\in \HH_1\times\ldots\times\HH_m\Big) \notag\\
&\sum_{i=1}^m \scal{C_i(x_1,\ldots, x_m) - 
C_i(\overline{x}_{1,1},\ldots,\overline{x}_{1,m}) }{x_i- \overline{x}_{1,i}} \geq 
\phi_j(\|x_j-\overline{x}_{1,j}\|),
\end{alignat}
then $x_{1,n}^j\to\overline{x}_{1,j}$.
\item 
\label{t:3vv} 
Suppose that $D^{-1}_j$ is uniformly monotone at $\overline{v}_{1,j}$, for some 
$j\in\{1,\ldots,  k\}$, then $v_{1,n}^j \to \overline{v}_{1,j}$. 
\item 
\label{t:3viv} 
Suppose that $B^{-1}_j$ is uniformly monotone at $\overline{v}_{2,j}$, for some 
$j\in\{1,\ldots,  k\}$, then $v_{2,n}^j \to \overline{v}_{2,j}$. 
\end{enumerate}
\end{theorem}
\begin{proof}
Let us introduce the Hilbert direct sums 
\begin{equation}
\HHH = \HH_1\oplus\ldots\oplus\HH_m, \quad \GGG = \GG_1\oplus\ldots\oplus\GG_s, 
\quad \YYY = \YY_1\oplus\ldots\oplus\YY_s, 
\quad  \XXX = \XX_1\oplus\ldots\oplus\XX_s.
\end{equation}
We use the boldsymbol to indicate the elements in these spaces. 
The scalar products and the norms of 
these spaces are defined in normal way. For example, in $\HHH$, 
\begin{equation}
 \scal{\cdot}{\cdot}\colon (\xx,\yy)\mapsto 
\sum_{i=1}^m \scal{x_i}{y_i}
\quad  \text{and}\quad
\|\cdot\|\colon \xx\mapsto \sqrt{\scal{\xx}{\xx}}.
\end{equation}
Set
\begin{equation}
\label{e:acl}
\begin{cases}
 \AAA\colon \HHH\to 2^{\HHH}\colon \xx \mapsto \cart_{i=1}^m A_ix_i\\
\CCC \colon \HHH\to \HHH\colon \xx \mapsto  (C_i\xx)_{1\leq i\leq m}\\ 
\LL\colon \HHH\to\GGG\colon \xx \mapsto \big(\sum_{i=1}^mL_{k,i}x_i\big)_{1\leq k\leq s}\\
\NNN\colon\GGG\to\XXX\colon \vv \mapsto (N_kv_k)_{1\leq k\leq s}\\
\zz = (z_1,\ldots, z_m), 
\end{cases}
\quad \text{and}\quad
\begin{cases}
 \BB\colon \YYY\to 2^{\YYY}\colon \vv \mapsto \cart_{k=1}^s B_kv_k\\
\DD\colon \XXX\to 2^{\XXX}\colon \vv \mapsto \cart_{k=1}^s D_kv_k\\
\MM\colon\GGG\to\YYY\colon \vv \mapsto (M_kv_k)_{1\leq k\leq s}\\
\rr = (r_1,\ldots, r_s).
\end{cases}
\end{equation}
Then, it follows  from \eqref{coco}  that 
\begin{equation}
\label{e:coso}
 (\forall(\xx,\yy)\in\HHH^2)\quad 
 \|\CCC\xx-\CCC\yy \|\leq \nu_{0}\|\xx-\yy\|
\quad 
\text{and}
\quad 
\scal{\CCC\xx-\CCC\yy}{\xx-\yy}\geq0,
\end{equation}
which shows that $\CCC$ is
$\nu_0$-Lipschitzian and monotone
hence they are
maximally monotone \cite[Corollary 20.25]{livre1}.
Moreover, it follows from \cite[Proposition 20.23]{livre1} 
that $\AAA$, $\BB$ and $\DD$  are maximally monotone. 
Furthermore, 
\begin{equation}
\label{e:conj}
\begin{cases}
 \LL^*\colon \GGG\to\HHH\colon \vv \mapsto \bigg(\sum_{k=1}^s L_{k,i}^*v_k\bigg)_{1\leq i\leq m}\\
\MM^*\colon\YYY\to\GGG\colon\vv\mapsto (M_{k}^*v_k)_{1\leq k\leq s}\\
\NNN^*\colon\XXX\to\GGG\colon\vv\mapsto (N_{k}^*v_k)_{1\leq k\leq s}.\\
\end{cases}
\end{equation}
Then, using \eqref{e:acl} and \eqref{e:conj},
we can rewrite the system of monotone inclusions \eqref{dualin} as
 monotone inclusions in $\KKK = \HHH\oplus\GGG$,
\begin{equation}
\label{e:pr}
 \text{find $(\overline{\xx},\overline{\vv})\in\KKK$ such that }\; 
\begin{cases}
\zz-\LL^*\overline{\vv}\in(\AAA+\CCC)\overline{\xx} \\ 
\LL\overline{\xx}-\rr\in \big((\MM^*\circ\BB\circ\MM)^{-1} 
+ (\NNN^*\circ\DD\circ\NNN)^{-1}\big)\overline{\vv}.
\end{cases}
\end{equation}
It follows from \eqref{e:pr} that there exists 
$\overline{\yy} \in \GGG$  such that 
\begin{equation}
\begin{cases}
\zz-\LL^*\overline{\vv}\in(\AAA+\CCC)\overline{\xx} \\
\overline{\yy} \in  (\MM^*\circ\BB\circ\MM)^{-1}\overline{\vv}\\
\LL\overline{\xx}-\overline{\yy} -\rr\in  (\NNN^*\circ\DD\circ\NNN)^{-1}\overline{\vv}
\end{cases}
\Leftrightarrow
\begin{cases}
 \zz-\LL^*\overline{\vv}\in(\AAA+\CCC)\overline{\xx} \\
\overline{\vv} \in  \MM^*\circ\BB\circ\MM\overline{\yy}\\
\overline{\vv} \in  \NNN^*\circ\DD\circ\NNN(\LL\overline{\xx}-\overline{\yy} -\rr),
 \end{cases}
\end{equation}
which implies that
\begin{equation}
\label{e:ess}
\begin{cases}
 \zz\in (\AAA+\CCC)\overline{\xx} 
+ \LL^*\NNN^*\big(\DD(\NNN\LL\overline{\xx}-\NNN\overline{\yy}-\NNN\rr)\big)\\
0\in\MM^*\circ\BB\circ\MM\overline{\yy}
-\NNN^*\big(\DD(\NNN\LL\overline{\xx}-\NNN\overline{\yy}-\NNN\rr)\big).
\end{cases}
\end{equation}
Since $\Omega \not=\emp$, the problem \eqref{e:ess} possesses at least one solution. 
The problem \eqref{e:ess} is a special case of
the primal problem in 
\cite[Eq.(1.2)]{Combettes13} with 
\begin{equation}
\label{e:spec}
 \begin{cases}
m =2, K =2,\\
\HHH_1 = \HHH,
 \GGG_1 = \XXX,\\
\HHH_2 =\GGG,
\GGG_2 = \YYY,\\
\zz_1 = \zz, ,\zz_2=0,\\
\rr_1 = \NNN\rr,\rr_2 =0,
 \end{cases}
\begin{cases}
 \LL_{1,1} = \NNN\LL,\\
 \LL_{1,2} = -\NNN,\\
\LL_{2,1} = 0,\\ 
\LL_{2,2} = \MM, 
 \end{cases}
\begin{cases}
 \AAA_1 = \AAA,\\ 
 \CCC_1 = \CCC,\\
\AAA_2 = 0,\\
\CCC_2 = 0,\\
\end{cases}
\quad \text{and}\quad
\begin{cases}
 \BBB_1 = \DD, \\
 \DD^{-1}_1 = 0,\\
\BB_2 = \BB,\\
\DD^{-1}_2 = 0. \\
\end{cases}
\end{equation}
It follows from \cite[Proposition 23.16]{livre1} that
\begin{equation}
\label{e:resol}
 (\forall \xx\in\HHH)(\gamma\in\left]0,+\infty\right[)\quad 
J_{\gamma\AAA_1}\xx = (J_{\gamma A_i}x_i)_{1\leq i\leq m}
\end{equation}
and
\begin{equation}
\label{e:resol1}
 (\forall \vv\in\XXX)(\gamma\in\left]0,+\infty\right[)\quad 
J_{\gamma\BBB_1}\vv = (J_{\gamma D_k}v_k)_{1\leq k\leq s}
\; \text{and}\;
 (\forall \vv\in\YYY)\quad J_{\gamma\BB_2}\vv = (J_{\gamma B_k}v_k)_{1\leq k\leq s}.
\end{equation}
Let us set 
\begin{equation}
\label{e:seq}
(\forall n\in\NN)
 \begin{cases}
\aaa_{1,1,n} = (a_{1,1,n}^1,\ldots, a_{1,1,n}^m)\\
\bb_{1,1,n} = (b_{1,1,n}^1,\ldots, b_{1,1,n}^m)\\
\cc_{1,1,n} = (c_{1,1,n}^1,\ldots, c_{1,1,n}^m)\\
\aaa_{1,2,n} = (a_{1,2,n}^1,\ldots, a_{1,2,n}^s)\\
\cc_{1,2,n} = (c_{1,2,n}^1,\ldots, c_{1,2,n}^s)\\
\end{cases}
\quad \text{and}\quad 
(\forall n\in\NN)
 \begin{cases}
\aaa_{2,1,n} = (a_{2,1,n}^1,\ldots, a_{2,1,n}^s)\\
\cc_{2,1,n} = (c_{2,1,n}^1,\ldots, c_{2,1,n}^s)\\
\aaa_{2,2,n} = (a_{2,2,n}^1,\ldots, a_{2,2,n}^s)\\
\bb_{2,2,n} = (b_{2,2,n}^1,\ldots, b_{2,2,n}^s)\\
\cc_{2,2,n} = (c_{2,2,n}^1,\ldots, c_{2,2,n}^s).\\
\end{cases}
\end{equation}
Then it follows from our assumptions that 
every sequence defined in \eqref{e:seq} is absolutely summable.
Moreover, upon setting
\begin{equation}
(\forall n\in\NN)
\begin{cases}
 \xx_{1,n} = (x_{1,n}^1,\ldots, x^{m}_{1,n})\\
\xx_{2,n} = (x_{2,n}^1,\ldots, x_{2,n}^s)\\
\end{cases}
\quad 
\text{and}
\quad 
\begin{cases}
 \vv_{1,n} = (v_{1,n}^1,\ldots, v^{s}_{1,n})\\
\vv_{2,n} = (v_{2,n}^1,\ldots, v_{2,n}^s)\\
\end{cases}
\end{equation}
and
\begin{equation}
 \label{e:seqs}
(\forall n\in\NN)
 \begin{cases}
\sss_{1,1,n} = (s_{1,1,n}^1,\ldots, s_{1,1,n}^m)\\
\pp_{1,1,n} = (p_{1,1,n}^1,\ldots, p_{1,1,n}^m)\\
\qq_{1,1,n} = (q_{1,1,n}^1,\ldots, q_{1,1,n}^m)\\
\pp_{1,2,n} = (p_{1,2,n}^1,\ldots, p_{1,2,n}^s)\\
\qq_{1,2,n} = (q_{1,2,n}^1,\ldots, q_{1,2,n}^s)\\
\end{cases}
\quad \text{and}\quad 
(\forall n\in\NN)
 \begin{cases}
\sss_{2,1,n} = (s_{2,1,n}^1,\ldots, s_{2,1,n}^s)\\
\pp_{2,1,n} = (p_{2,1,n}^1,\ldots, p_{2,1,n}^s)\\
\qq_{2,1,n} = (q_{2,1,n}^1,\ldots, q_{2,1,n}^s)\\
\sss_{2,2,n} = (s_{2,2,n}^1,\ldots, s_{2,2,n}^s)\\
\pp_{2,2,n} = (p_{2,2,n}^1,\ldots, p_{2,2,n}^s)\\
\qq_{2,2,n} = (q_{2,2,n}^1,\ldots, q_{2,2,n}^s)\\
\end{cases}
\end{equation}
and in view of \eqref{e:acl},\eqref{e:conj}, \eqref{e:spec} and \eqref{e:resol}, \eqref{e:resol1}, 
algorithm \eqref{e:forwardback} reduces to a special case of the algorithm in 
\cite[Eq. (2.4)]{Combettes13}.
Moreover, it follows from  \eqref{e:cons2} and \eqref{e:spec}  
that the condition \cite[Eq.(1.1)]{Combettes13} is satisfied. Furthermore,
the conditions on stepsize $(\gamma_n)_{n\in\NN}$ and,
as shown above, every specific conditions on 
operators and the error sequences are also satisfied. To sum up,
every specific conditions in \cite[Problem 1.1]{Combettes13} and  
\cite[Theorem 2.4]{Combettes13} are satisfied.
 
\ref{t:2i}\ref{t:2ii}: These conclusions 
follow from \cite[Theorem 2.4(i)]{Combettes13} 
and \cite[Theorem 2.4(ii)]{Combettes13}, respectively.

\ref{t:2iii}: It follows from \cite[Theorem 2.4(iii)(c)]{Combettes13} 
and \cite[Theorem 2.4(iii)(d)]{Combettes13}
that $\xx_{1,n}\weakly \overline{\xx}_1$, $\xx_{2,n}\weakly \overline{\xx}_2 $ and 
$\vv_{1,n}\weakly \overline{\vv}_1 $, $\vv_{2,n}\weakly \overline{\vv}_2$, 
We next derive from \cite[Theorem 2.4(iii)(a)]{Combettes13} 
and \cite[Theorem 2.4(iii)(b)]{Combettes13} that,
for every $i\in\{1\ldots,m\}$ and  $k\in\{1\ldots,s\}$,
 \begin{equation}
\label{e:inclusion1}
z_i- \sum_{k=1}^s L_{k,i}^*N_{k}^*\overline{v}_{1,k}\in A_i\overline{x}_{1,i} 
+ C_i(\overline{x}_{1,1},\ldots, \overline{x}_{1,m})
\quad \text{and}\quad 
  M^{*}_k\overline{v}_{2,k} = N^{*}_k\overline{v}_{1,k}.
\end{equation}
and
\begin{equation}
\label{e:inclusion2}
N_k\bigg(\sum_{i=1}^m L_{k,i}\overline{x}_{1,i}  -r_k -\overline{x}_{2,k}\bigg)\in
D_{k}^{-1}\overline{v}_{1,k}\quad \text{and}\quad 
 M_k\overline{x}_{2,k}  \in B_{k}^{-1}\overline{v}_{2,k}.\\
\end{equation}
We have
\begin{alignat}{2}
\label{e:sss}
\eqref{e:inclusion2}& \Leftrightarrow \overline{v}_{1,k}\in 
D_{k}\bigg(N_k\bigg(\sum_{i=1}^m L_{k,i}\overline{x}_{1,i}  -r_k -\overline{x}_{2,k}\bigg) \bigg)
\quad \text{and}\quad 
\overline{v}_{2,k}\in B_{k}\big( M_k\overline{x}_{2,k} \big)\\
&\Rightarrow
N^{*}_k\overline{v}_{1,k}\in 
N^{*}_k\bigg(D_{k}\bigg(N_k\bigg(\sum_{i=1}^m L_{k,i}\overline{x}_{1,i}  -r_k -\overline{x}_{2,k}) 
\bigg)\bigg)\bigg)
\; \text{and}\; 
M^{*}_k\overline{v}_{2,k}\in M^{*}_k\big(B_{k}\big( M_k\overline{x}_{2,k}\big)\big)\notag\\
&\Rightarrow
\sum_{i=1}^m L_{k,i}\overline{x}_{1,i}  -r_k -\overline{x}_{2,k}
 \in (N^{*}_k\circ D_{k}\circ N_k)^{-1}(N^{*}_k\overline{v}_{1,k})
\; \text{and}\;  \overline{x}_{2,k}\in (M^{*}_k\circ B_{k}\circ M_k)^{-1}(M^{*}_k\overline{v}_{2,k})\notag\\
&\Rightarrow
\sum_{i=1}^m L_{k,i}\overline{x}_{1,i}  -r_k\in (N^{*}_k\circ D_{k}\circ N_k)^{-1}(N^{*}_k\overline{v}_{1,k})
+(M^{*}_k\circ B_{k}\circ M_k)^{-1}(N^{*}_k\overline{v}_{1,k}).
\end{alignat}
Therefore, \eqref{e:inclusion1} and \eqref{e:sss} shows that 
$(\overline{x}_{1,1},\ldots, \overline{x}_{1,m}, N^{*}_1\overline{v}_{1,1},\ldots, N^{*}_s\overline{v}_{1,s})$
is a solution to \eqref{dualin}.

\ref{t:3iv}:
For every $n\in\NN$ and every $i\in\{1,\ldots,m\}$ and $k\in\{1,\ldots,s\}$, set
\begin{equation}
\label{e:sva21}
 \begin{cases}
\widetilde{s}_{1,1,n}^i = x^{i}_{1,n} - \gamma_n\big(C_i(x_{1,n}^1,\ldots, x_{1,n}^m)\\
\hspace{3cm}+ \sum_{k=1}^{s}L_{k,i}^*N_{k}^*v_{1,n}^k\big)\\
\widetilde{p}_{1,2,n}^k = x_{2,n}^k -\gamma_n\big( N_{k}^*v_{1,n}^k - M_{k}^*v_{2,n}^k \big)\\
\widetilde{p}_{1,1,n}^i=J_{\gamma_n A_i}(\widetilde{s}_{1,1,n}^i+ \gamma_nz_i)\\
\end{cases}
\quad
\text{and} 
\quad
\begin{cases}
\widetilde{s}_{2,1,n}^k = v_{1,n}^k + \gamma_n\big( 
\sum_{i=1}^mN_kL_{k,i} x_{1,n}^i - N_kx_{2,n}^k  \big)\\
\widetilde{p}_{2,1,n}^k = \widetilde{s}_{2,1,n}^k-\gamma_n \big(N_kr_k\\
\hspace{3cm}+ J_{\gamma_{n}^{-1}D_k}(\gamma_{n}^{-1}\widetilde{s}_{2,1,n}^k -N_kr_k)\big)\\
\widetilde{s}_{2,2,n}^k = v_{2,n}^k + \gamma_nM_kx_{2,n}^k\\
\widetilde{p}_{2,2,n}^k 
= \widetilde{s}_{2,2,n}^k -\gamma_nJ_{\gamma^{-1}_nB_{k}}(\gamma^{-1}_n\widetilde{s}_{2,2,n}^k).\\
\end{cases}
\end{equation}
Since $(\forall i\in\{1,\ldots,m\})\; a_{1,1,n}^i\to 0, b_{1,1,n}^i\to 0$, 
$(\forall k\in\{1,\ldots,s\})\;a_{2,1,n}^k\to 0, a_{2,2,n}^k\to 0$ and
$b_{2,1,n}^k\to 0, b_{2,2,n}^k\to 0$ and since the resolvents of
$(A_i)_{1\leq i\leq m}$, $(B_{k}^{-1})_{1\leq k\leq s}$
and $(D_{k}^{-1})_{1\leq k\leq s}$ are nonexpansive, we obtain
\begin{equation}
\begin{cases} 
(\forall i\in\{1,\ldots, m\})\;
\widetilde{p}_{1,1,n}^i - p_{1,1,n}^i\to 0\\
(\forall k\in\{1,\ldots, s\})\;
\widetilde{p}_{1,2,n}^k-p_{1,2,n}^k\to 0
\end{cases}
\quad \text{and}\quad
\begin{cases}
(\forall k\in \{1,\ldots,s\})
\;\widetilde{p}_{2,1,n}^k-p_{2,1,n}^k \to 0\\
(\forall k\in \{1,\ldots,s\})\;
\widetilde{p}_{2,2,n}^k-p_{2,2,n}^k \to 0.
\end{cases}
\end{equation}
In turn, by \ref{t:2i} and \ref{t:2ii}, we obtain
\begin{equation}
\label{e:hoaquangoc}
\begin{cases}
(\forall i\in\{1,\ldots,m\})\quad
\widetilde{p}_{1,1,n}^i - x_{1,n}^i\to 0, 
\quad \widetilde{p}_{1,1,n}^i \weakly \overline{x}_{1,i}\\
(\forall k\in\{1,\ldots,s\})\quad
\widetilde{p}_{1,2,n}^k-p_{1,2,n}^k \to 0,
\quad \widetilde{p}_{1,2,n}^k\weakly \overline{x}_{2,k}\\
\end{cases}
\end{equation}
and 
\begin{equation}
\label{e:fffd}
(\forall k\in\{1,\ldots,s\})\quad
\begin{cases}
 \widetilde{p}_{2,1,n}^k - v_{1,n}^k\to 0,\quad \widetilde{p}_{2,1,n}^k \weakly \overline{v}_{1,k}\\
\widetilde{p}_{2,2,n}^k-v_{2,n}^k \to 0,\quad \widetilde{p}_{2,2,n}^k\weakly \overline{v}_{2,k}.\\
\end{cases}
\end{equation}
Set 
\begin{equation}
\label{e:minus}
(\forall n\in\NN)\quad
 \begin{cases}
  \widetilde{\pp}_{1,1,n}= (\widetilde{p}_{1,1,n}^1,\ldots, \widetilde{p}_{1,1,n}^m)\\
 \widetilde{\pp}_{1,2,n}= (\widetilde{p}_{1,2,n}^1,\ldots,\widetilde{p}_{1,2,n}^s)\\
 \end{cases}
\quad \text{and}\quad 
 \begin{cases}
  \widetilde{\pp}_{2,1,n}= (\widetilde{p}_{2,1,n}^1,\ldots, \widetilde{p}_{2,1,n}^s)\\
  \widetilde{\pp}_{2,2,n}= (\widetilde{p}_{2,2,n}^1,\ldots, \widetilde{p}_{2,2,n}^s).\\
 \end{cases}
\end{equation}
Then, it follows from \eqref{e:fffd} that 
\begin{equation}
\label{e:vn1}
\begin{cases} 
\gamma_{n}^{-1}(\xx_{1,n}-\widetilde{\pp}_{1,1,n})\to 0\\
\gamma_{n}^{-1}(\xx_{2,n}-\widetilde{\pp}_{1,2,n})\to 0\\
\end{cases}
\quad \text{and}\quad 
\begin{cases} 
\gamma_{n}^{-1}(\vv_{1,n}-\widetilde{\pp}_{2,1,n})\to 0\\
\gamma_{n}^{-1}(\vv_{2,n}-\widetilde{\pp}_{2,2,n})\to 0.\\
\end{cases}
\end{equation}
Furthermore, we derive from
\eqref{e:sva21} that, 
for every $i\in \{1,\ldots,m\}$ and  $k\in \{1,\ldots,s\}$
\begin{equation}
\label{e:dualsola}
(\forall n\in\NN)\quad
 \begin{cases}
 \gamma^{-1}_n(x^{i}_{1,n}-\widetilde{p}_{1,1,n}^i)
-\sum_{k = 1}^s L_{k,i}^*N_{k}^*v_{1,n}^k - C_i(x_{1,n}^1,\ldots,x_{1,n}^m) \in -z_i + A_i\widetilde{p}_{1,1,n}^i\\ 
\gamma^{-1}_n(\widetilde{s}_{2,2,n}^k-\widetilde{p}_{2,2,n}^k)  
\in  B_{k}^{-1}\widetilde{p}_{2,2,n}^k\\
\gamma^{-1}_n(\widetilde{s}_{2,1,n}^k-\widetilde{p}_{2,1,n}^k)  
\in r_k+ D_{k}^{-1}\widetilde{p}_{2,1,n}^k.\\
 \end{cases}
\end{equation}
Since $A_j$ is uniformly monotone at $\overline{x}_{1,j}$, 
using \eqref{e:dualsola} and \eqref{e:inclusion1},
there exists an increasing function
$\phi_{A_j}\colon\left[0,+\infty\right[\to\left[0,+\infty\right]$ 
vanishing only at $0$ such that, for every $n\in\NN$,  
\begin{alignat}{2}
\phi_{A_j}(\|\widetilde{p}_{1,1,n}^j -\overline{x}_{1,j}\|) 
&\leqslant\scal{\widetilde{p}_{1,1,n}^j
-\overline{x}_{1,j}}{\gamma_{n}^{-1}(x_{1,n}^j-\widetilde{p}_{1,1,n}^j) 
-\sum_{k = 1}^sL_{k,j}^*N_{k}^*(v_{1,n}^k-\overline{v}_{1,k})
-(C_j\xx_{1,n}-C_j\bar{\xx}_{1})}\notag\\
&= \scal{\widetilde{p}_{1,1,n}^j-\overline{x}_{1,j}}{\gamma_{n}^{-1}(x_{1,n}^j
-\widetilde{p}_{1,1,n}^j}-\sum_{k = 1}^s\scal{\widetilde{p}_{1,1,n}^j
-\overline{x}_{1,j}}{L_{k,j}^*N_{k}^*(v_{1,n}^k - \overline{v}_{1,k})}\notag\\
&\quad-\chi_{j,n}\label{cz1},
\end{alignat}
where we denote $\big(\forall n\in\NN\big)\; 
\chi_{j,n} = \scal{\widetilde{p}_{1,1,n}^j 
-\bar{x}_{1,j}}{C_j\xx_{1,n} -C_j\bar{\xx}_{1}}$.
Therefore,
\begin{alignat}{2}
\label{e:add4}
 \phi_{A_j}(\|\widetilde{p}_{1,1,n}^j-\overline{x}_{1,j}\|)&\leq 
\scal{\widetilde{\pp}_{1,1,n} -\overline{\xx}_{1}}{\gamma_{n}^{-1}(\xx_{1,n}-\widetilde{\pp}_{1,1,n}}- 
\scal{\widetilde{\pp}_{1,1,n}
-\overline{\xx}_{1}}{\LL^*\NNN^*(\vv_{1,n} - \overline{\vv}_{1})}-\chi_{n}\notag\\
&=\scal{\widetilde{\pp}_{1,1,n} -\overline{\xx}_{1}}{\gamma_{n}^{-1}(\xx_{1,n}-\widetilde{\pp}_{1,1,n}}
- \scal{\widetilde{\pp}_{1,1,n}-\xx_{1,n}}{\LL^*\NNN^*(\vv_{1,n} - \overline{\vv}_{1})}\notag\\
&\quad-\scal{\xx_{1,n}-\overline{\xx}_1}{\LL^*\NNN^*(\vv_{1,n} - \overline{\vv}_{1})}-\chi_{n},
\end{alignat}
where $\chi_n = \sum_{i=1}^m\chi_{i,n} = 
\scal{\widetilde{\pp}_{1,1,n} -\bar{\xx}_{1}}{\CCC\xx_{1,n} -\CCC\bar{\xx}_{1}}$. 
Since $(B_{k}^{-1})_{1\leq k\leq s}$ and $(D_{k}^{-1})_{1\leq k\leq s}$  are monotone,
we derive from \eqref{e:inclusion2} and \eqref{e:dualsola} that for every $ k\in\{1,\ldots,s\}$,
\begin{equation}  
 \begin{cases}
0\leq\scal{\widetilde{p}_{2,1,n}^k-\overline{v}_{1,k}}{\gamma_{n}^{-1}(v_{1,n}^k
-\widetilde{p}_{2,1,n}^k) +\sum_{i=1}^m N_kL_{k,i}(x_{1,n}^i-\overline{x}_{1,i})-N_k(x_{2,n}^k-\overline{x}_{2,k})}\\
0\leq\scal{\widetilde{p}_{2,2,n}^k-\overline{v}_{2,k}}{\gamma_{n}^{-1}(v_{2,n}^k-\widetilde{p}_{2,2,n}^k)
+M_k(x_{2,n}^k-\overline{x}_{2,k})},
\end{cases}
\end{equation}
 which implies that 
\begin{equation}
\label{e:add2}
0\leq \scal{\widetilde{\pp}_{2,2,n} 
-\overline{\vv}_{2}}{\gamma_{n}^{-1}(\vv_{2,n}-\widetilde{\pp}_{2,2,n})}
+\scal{\widetilde{\pp}_{2,2,n} -\overline{\vv}_{2}}{ \MM(\xx_{2,n}-\overline{\xx}_{2})}\\
\end{equation}
and 
\begin{alignat}{2}
\label{e:add1}
0&\leq \scal{\widetilde{\pp}_{2,1,n}
-\overline{\vv}_{1}}{\gamma_{n}^{-1}(\vv_{1,n}-\widetilde{\pp}_{2,1,n}) }
+\scal{\NNN\LL(\xx_{1,n}-\overline{\xx}_1)}{\widetilde{\pp}_{2,1,n} 
-\overline{\vv}_{1}}\notag\\
&\hspace{7cm}-\scal{\widetilde{\pp}_{2,1,n}
-\overline{\vv}_{1}}{ \NNN(\xx_{2,n}-\overline{\xx}_{2})}.
\end{alignat}
We expand $(\chi_n)_{n\in\NN}$ as
\begin{alignat}{2}
\label{e:expa}
(\forall n\in\NN)\quad
\chi_n &= \scal{\xx_{1,n}-\overline{\xx}_1}{\CCC\xx_{1,n}-\CCC\overline{\xx}_1}
 +\scal{\widetilde{\pp}_{1,1,n} -\xx_{1,n}}{ \CCC\xx_{1,n}-\CCC\overline{\xx}}\notag\\
&\geq \scal{\widetilde{\pp}_{1,1,n} -\xx_{1,n}}{ \CCC\xx_{1,n}-\CCC\overline{\xx}},
\end{alignat}
where the last inequality follows from the monotonicity of $\CCC$. 
Now, adding~\eqref{e:add1}, \eqref{e:add2}, \eqref{e:add4}, \eqref{e:expa} 
and using $\MM^*\overline{\vv}_2 = \NNN^*\overline{\vv}_1$,
 we obtain, 
\begin{alignat}{2}
\label{e:vn3}
 \phi_{A_j}(\|\widetilde{p}_{1,1,n}^j -\overline{x}_{1,j}\|)&\leq 
\scal{\widetilde{\pp}_{1,1,n} -\overline{\xx}_{1}}{\gamma_{n}^{-1}(\xx_{1,n}-\widetilde{\pp}_{1,1,n}}
- \scal{\widetilde{\pp}_{1,1,n}-\xx_{1,n}}{\LL^*\NNN^*(\vv_{1,n} - \overline{\vv}_{1})}\notag\\
&\quad+ \scal{\widetilde{\pp}_{2,2,n} 
-\overline{\vv}_{2}}{\gamma_{n}^{-1}(\vv_{2,n}-\widetilde{\pp}_{2,2,n})}
+\scal{\widetilde{\pp}_{2,1,n}
-\overline{\vv}_{1}}{\gamma_{n}^{-1}(\vv_{1,n}-\widetilde{\pp}_{2,1,n})}\notag\\
&\quad+\scal{\MM^*\widetilde{\pp}_{2,2,n} - \NNN^*\widetilde{\pp}_{2,1,n}}{\xx_{2,n}-\overline{\xx}_{2}}
+\scal{\NNN\LL(\xx_{1,n}-\overline{\xx}_1)}{\widetilde{\pp}_{2,1,n} 
-\vv_{1,n}}\notag\\
&\quad -\chi_n.
\end{alignat}
We next derive from \eqref{e:forwardback}   that 
\begin{alignat}{2}
(\forall k\in\{1,\ldots,s\})&\quad
M^{*}_kp_{2,2,n}^k- N^{*}_kp_{2,1,n}^k= \gamma_{n}^{-1}(p_{1,2,n}^k -q_{1,2,n}^k) + c_{1,2,n}^k,
\end{alignat}
which and \eqref{e:minus}, \eqref{e:vn1}, and \cite[Theorem 2.5(i)]{siop2} imply that 
\begin{equation}
 \MM^{*}\widetilde{\pp}_{2,2,n}^k- \NNN^{*}\widetilde{\pp}_{2,1,n}^k\to 0.
\end{equation}
Furthermore, since $((\xx_{i,n})_{n\in\NN})_{1\leq i\leq 2}$ and 
$(\widetilde{\pp}_{1,1,n})_{n\in\NN}$, 
$(\widetilde{\pp}_{2,1,n})_{n\in\NN}$, 
$(\widetilde{\pp}_{2,2,n})_{n\in\NN}$,
$(\vv_{1,n})_{n\in\NN}$ 
 converge weakly, they are bounded. Hence 
\begin{equation}
\label{e:vn}
\tau = \sup_{n\in\NN}\{\max_{1\leq i\leq2}
\{\|\xx_{i,n}-\overline{\xx}_i\|,\|\widetilde{\pp}_{2,i,n}-\overline{\vv}_i\|
,\|\widetilde{\pp}_{1,1,n}-\overline{\xx}_1\|\},
\|\vv_{1,n}-\overline{\vv}_1\|
 \} <+\infty.
\end{equation}
Then, using Cauchy-Schart, the Lipchitzianity of $\CCC$ and 
\eqref{e:vn}, \eqref{e:vn1}, 
it follows from \eqref{e:vn3} that
\begin{alignat}{2}
\label{e:vn33}
\phi_{A_j}(\|\widetilde{p}_{1,1,n}^j -\overline{x}_{1,j}\|)&\leq 
\tau\bigg(\big(\gamma_{n}^{-1}+\|\NNN\LL\|\big)\big(\|\widetilde{\pp}_{1,1,n} -\overline{\xx}_{1} \|+ 
\|\widetilde{\pp}_{2,1,n} -\vv_{1,n}\|\big) +\|\gamma_{n}^{-1}(\vv_{2,n}-\widetilde{\pp}_{2,2,n})\| \notag\\
&\quad
+\|\gamma_{n}^{-1}(\vv_{1,n}-\widetilde{\pp}_{2,1,n})\|
+ \mu\|\widetilde{\pp}_{1,1,n} -\overline{\xx}_{1} \| \bigg)\notag\\
&\quad\to 0,
\end{alignat}
in turn, $\widetilde{p}_{1,1,n}^j \to\overline{x}_{1,j}$ 
and hence, by \eqref{e:hoaquangoc}, $x_{1,n}^j\to \overline{x}_{1,j}$.
 
\ref{t:3v}: Since $\CCC$ is uniformly at $\overline{\xx}_1$, there exists an 
increasing function $\phi_{\CCC}\colon \left[0,+\infty\right[\to \left[0,+\infty\right]$
vanishing only at $0$
such that
\begin{equation}
(\forall n\in\NN)\quad \scal{\xx_{1,n}-\overline{\xx}_1}{\CCC\xx_{1,n}- \CCC\overline{\xx}_1}
\geq \phi_{\CCC}(\|\xx_{1,n}-\overline{\xx}_1\|),
\end{equation}
and hence, \eqref{e:expa} becomes
\begin{alignat}{2}
\label{e:expaa}
(\forall n\in\NN)\quad
\chi_n &= \scal{\xx_{1,n}-\overline{\xx}_1}{\CCC\xx_{1,n}- \CCC\overline{\xx}_1}
 + \scal{\widetilde{\pp}_{1,1,n} -\xx_{1,n}}{ \CCC\xx_{1,n}- \CCC\overline{\xx}}\notag\\
&\geq \scal{\widetilde{\pp}_{1,1,n} -\xx_{1,n}}{ \CCC\xx_{1,n}- \CCC\overline{\xx}}
+ \phi_{\CCC}(\|\xx_{1,n}-\overline{\xx}_1\|).
\end{alignat}
Processing as in \ref{t:3iv}, \eqref{e:vn33} becomes 
\begin{alignat}{2}
\label{e:sds1}
 \phi_{\CCC}(\|\xx_{1,n}-\overline{\xx}_1\|)
&\leq 
\tau\bigg(\big(\gamma_{n}^{-1}+\|\NNN\LL\|\big)\big(\|\widetilde{\pp}_{1,1,n} -\overline{\xx}_{1} \|+ 
\|\widetilde{\pp}_{2,1,n} -\vv_{1,n}\|\big) +\|\gamma_{n}^{-1}(\vv_{2,n}-\widetilde{\pp}_{2,2,n})\| \notag\\
&\quad
+\|\gamma_{n}^{-1}(\vv_{1,n}-\widetilde{\pp}_{2,1,n})\|
+ \mu\|\widetilde{\pp}_{1,1,n} -\overline{\xx}_{1} \| \bigg)\notag\\
&\quad\to 0,
\end{alignat}
in turn, $\xx_{1,n}\to\overline{\xx}_1$ or equivalently 
$(\forall i\in\{1,\ldots,m\})\; x_{1,n}^i\to\overline{x}_{1,i}$.

\ref{t:3vi}: Using the same argument as in the proof of \ref{t:3v}, 
we reach at \eqref{e:sds1} where 
$\phi_{\CCC}(\|\xx_{1,n}-\overline{\xx}_1\|)$ is replaced by 
$\phi_{j}(\|x_{1,n}^j-\overline{x}_{1,j}\|)$, 
and hence we obtain the conclusion.

\ref{t:3viv}\&\ref{t:3vv}: Using the same argument as in the proof of \ref{t:3v}.   
\end{proof}
\begin{remark} Here are some remarks.
\begin{enumerate}
\item In the special case when $m=1$ and $(\forall k\in\{1,\ldots,s\})\;\GG_k=\HH_1, L_{k,i} =\Id$,
algorithm \eqref{e:forwardback} reduces to the recent algorithm proposed in \cite[Eq.(3.15)]{plc13b} where 
the convergence results are proved under the same conditions.
\item In the special case when $m=1$ and $C_1$ is restricted to be cocoercive, i.e, $C_{1}^{-1}$ is strongly 
monotone, an alternative algorithm proposed in \cite{Botb} can be used to solve Problem \ref{prob}.
 \item In the case when $(\forall k\in\{1\ldots,s\})(\forall i\in\{1,\ldots,m\})\; L_{k,i} =0$, 
algorithm \eqref{e:forwardback} is separated into two different 
algorithms which solve respectively the first $m$ inclusions
and the last $k$ inclusions in \eqref{dualin} independently. 
\item Condition \eqref{csoco} is satisfied, for example, when each $C_i$ is restricted 
to be univariate and monotone, and $C_j$ is uniformly monotone.
\end{enumerate}
 \end{remark}
\section{Applications to minimization problems}
The algorithm proposed has a structure of the 
forward-backward-forward splitting as in \cite{plc13b,siop2,Combettes13,plc6,Tseng00}.
The applications of this type of algorithm to specific problems in applied mathematics can be found in 
\cite{livre1,plc13b,Luis11,siop2,Combettes13,plc6,anna,Tseng00} and references therein.
We provide an application to the following minimization problem which extends  
\cite[Problem 4.1]{plc13b} and \cite[Problem 4.1]{Botb}. We recall that 
the infimal convolution of the two functions $f$ and $g$ from $\HH$ to $\left]-\infty,+\infty\right]$ is 
\begin{equation}
 f\;\vuo\; g\colon x \mapsto \inf_{y\in\HH}(f(y)+g(x-y)).
\end{equation}
\begin{problem}
\label{prob1}
Let $m,s$ be strictly positive integers.
For every $i\in \{1,\ldots,m\}$, 
let $(\HH_i,\scal{\cdot}{\cdot})$ be a real Hilbert space,
let $z_i\in\HH_i$, let $f_{i}\in\Gamma_0(\HH_i)$,
let $\varphi\colon \HH_1\times\ldots\times\HH_m\to\RR$ be 
convex differentiable function with $\nu_0$-Lipschitz continuous gradient 
$\nabla \varphi = (\nabla_1\varphi,\ldots,\nabla_m\varphi)$,
for 
some $\nu_0\in \left[0,+\infty\right[$. 
For every $k\in\{1,\ldots, s\}$, 
let $(\GG_k,\scal{\cdot}{\cdot})$, $(\YY_k,\scal{\cdot}{\cdot})$ and 
$(\XX_k,\scal{\cdot}{\cdot})$ be real Hilbert spaces, 
let $r_k \in \GG_k$, 
let $g_{k}\in\Gamma_0(\YY_k)$, let $\ell_k\in\Gamma_0(\XX_k)$,
let $M_k\colon \GG_k\to \YY_k$  and $N_k\colon \GG_k\to \XX_k$  
be  bounded linear operators.
For every $i\in\{1,\ldots, m\}$ and every $k\in\{1,\ldots,s\}$,
let $L_{k,i}\colon\HH_i \to\GG_k$ 
be a  bounded linear operator. 
The primal problems is to 
\begin{alignat}{2} 
\label{primal2}
&\underset{x_1\in\HH_1,\ldots, x_m\in\HH_m}{\text{minimize}}
 \sum_{k=1}^s\big((\ell_k\circ N_k)\;\vuo\; (g_{k}\circ M_k)
\big)\bigg(\sum_{i=1}^m L_{k,i}x_i -r_k\bigg)\notag\\
&\hspace{5cm} +\sum_{i=1}^{m}\big(f_i(x_i) - \scal{x_i}{z_i}\big)
+\varphi(x_1,\ldots,x_m),
\end{alignat}
and the dual problem is to 
\begin{alignat}{2}
\label{dual2} 
&\underset{v_1\in\GG_1,\ldots, v_s\in\GG_s}{\text{minimize}}
\bigg(\varphi^{*}\;\vuo\; 
\bigg(\sum_{i=1}^m f_{i}^*\bigg)\bigg)\bigg(\Big(z_i-\sum_{k=1}^s L_{k,i}^*v_k\Big)_{1\leq i\leq m} \bigg)\notag\\
&\hspace{4cm} 
 +\sum_{k=1}^{s}\bigg( (\ell_{k}\circ N_k)^*(v_k) + ( g_{k}\circ M_k)^*(v_k) + \scal{v_k}{r_k}\bigg).
\end{alignat}
\end{problem}
\begin{corollary}
In Problem \ref{prob1}, suppose that \eqref{e:cons2} is satisfied and
for every $ (k,i)\in\{1,\ldots,s\}\times\{1,\ldots,m\}$
\begin{equation}
\label{e:dion1}
0\in \sri\big(\dom (\ell_{k}\circ N_k)^*)-\dom (g_{k}\circ M_k)^{*}\big),
\end{equation}
and 
\begin{equation}
\label{e:dion2}
z_i\in\ran\bigg( \partial f_i+ \sum_{k=1}^s L_{k,i}^*\circ
\Big((N_{k}^*\circ(\partial \ell_k)\circ N_k) \;\vuo\; 
(M_{k}^*\circ (\partial g_{k})\circ M_k) 
\Big)\circ\bigg(\sum_{j=1}^m L_{k,j}\cdot-r_k\bigg) + 
\nabla_{i}\varphi\bigg).
\end{equation}
For every $i\in\{1,\ldots,m\}$,
let $(a^{i}_{1,1,n})_{n\in\NN}$, $(b^{i}_{1,1,n})_{n\in\NN}$, $(c^{i}_{1,1,n})_{n\in\NN}$ 
be absolutely 
summable sequences in $\HH_i$, 
for every $k\in\{1,\ldots,s\}$,
let $(a^{k}_{1,2,n})_{n\in\NN}$, $(c^{k}_{1,2,n})_{n\in\NN}$ be  
absolutely summable sequences in $\GG_k$,
let $(a^{k}_{2,1,n})_{n\in\NN}$
$(b^{k}_{2,1,n})_{n\in\NN}$, $(c^{k}_{2,1,n})_{n\in\NN}$ absolutely 
summable sequences in $\XX_k$,
$(a^{k}_{2,2,n})_{n\in\NN}$, $(b^{k}_{2,2,n})_{n\in\NN}$, $(c^{k}_{2,2,n})_{n\in\NN}$ 
be  absolutely summable sequences in $\YY_k$. For every $i\in\{1,\ldots,m\}$ and $k\in\{1,\ldots,s\}$,
let $x^{i}_{1,0} \in \HH_i$, $x_{2,0}^k \in \GG_k$ and 
$v_{1,0}^k \in \XX_k$, $v_{2,0}^k \in \YY_k$,  
let $\varepsilon \in \left]0,1/(\beta+1)\right[$, 
let $(\gamma_n)_{n\in\NN}$ be sequence in 
$\left[\varepsilon, (1-\varepsilon)/\beta \right]$ and set
\begin{equation}
\label{e:forwardback1}
\begin{array}{l}
\operatorname{For}\;n=0,1,\ldots,\\
\left\lfloor
\begin{array}{l}
\operatorname{For}\;i=1,\ldots, m\\
\left\lfloor
\begin{array}{l}
s_{1,1,n}^i = x_{1,n}^i -\gamma_n\big(\nabla_i\varphi(x_{1,n}^1,\ldots,x_{1,n}^m) 
+ \sum_{k=1}^sL_{k,i}^*N_{k}^*v_{1,n}^k + a_{1,1,n}^i\big)\\
p_{1,1,n}^i = \prox_{\gamma_n f_i}(s_{1,1,n}^i +\gamma_n z_i) + b_{1,1,n}^i \\
\end{array}
\right.\\[2mm]
\operatorname{For}\;k=1,\ldots,s\\
\left\lfloor
\begin{array}{l}
p_{1,2,n}^k = x_{2,n}^k +\gamma_n\big( 
 N_{k}^*v_{1,n}^k - M_{k}^*v_{2,n}^k + a_{1,2,n}^k \big)\\
s_{2,1,n}^k = v_{1,n}^k + \gamma_n\big( 
\sum_{i=1}^mN_kL_{k,i} x_{1,n}^i - N_kx_{2,n}^k + a_{2,1,n}^k \big)\\
p_{2,1,n}^k = s_{2,1,n}^k-\gamma_n \big(N_kr_k+ 
\prox_{\gamma_{n}^{-1}\ell_k}(\gamma_{n}^{-1}s_{2,1,n}^k -N_kr_k) + b_{2,1,n}^k \big)\\
q_{2,1,n}^k = p_{2,1,n}^k 
+\gamma_n\big( N_k\sum_{i=1}^mL_{k,i}p_{1,1,n}^i-N_kp_{1,2,n}^k +c_{2,1,n}^k \big) \\
v_{1,n+1}^k = v_{1,n}^k - s_{2,1,n}^k + q_{2,1,n}^k\\
s_{2,2,n}^k = v_{2,n}^k + \gamma_n\big( 
M_kx_{2,n}^k + a_{2,2,n}^k \big)\\
p_{2,2,n}^k = s_{2,2,n}^k-\gamma_n \big( 
J_{\gamma_{n}^{-1}g_k}(\gamma_{n}^{-1}s_{2,2,n}^k) + b_{2,2,n}^k \big)\\
q_{2,2,n}^k = p_{2,2,n}^k 
+\gamma_n\big( M_kp_{1,2,n}^k +c_{2,2,n}^k \big) \\
v_{2,n+1}^k = v_{2,n}^k - s_{2,2,n}^k + q_{2,2,n}^k\\
q_{1,2,n}^k = p_{1,2,n}^k +\gamma_n\big(
N_{k}^*p_{2,1,n}^k - M_{k}^*p_{2,2,n}^k + c_{1,2,n}^k \big)\\
x_{2,n+1}^k = x_{2,n}^k -p_{1,2,n}^k+ q_{1,2,n}^k\\
\end{array}
\right.\\[2mm]
\operatorname{For}\;i=1,\ldots, m\\
\left\lfloor
\begin{array}{l}
q_{1,1,n}^i = p_{1,1,n}^i -\gamma_n\big(\nabla_i\varphi(p_{1,1,n}^1,\ldots,p_{1,1,n}^m) 
+ \sum_{k=1}^sL_{k,i}^*N^{*}_kp_{2,1,n}^k + c_{1,1,n}^{i} \big)\\
x_{1,n+1}^i = x_{1,n}^i -s_{1,1,n}^i+ q_{1,1,n}^i.\\
\end{array}
\right.\\[2mm]
\end{array}
\right.\\[2mm]
\end{array}
\end{equation}
Then the following hold for each $i\in\{1,\ldots,m\}$ and $k\in\{1,\ldots,s\}$,
\begin{enumerate}
 \item \label{t:2i}
$\sum_{n\in\NN}\|x_{1,n}^i - p_{1,1,n}^i \|^2 < +\infty $\quad\text{and}\quad 
$\sum_{n\in\NN}\|x_{2,n}^k - p_{1,2,n}^k \|^2 < +\infty $.
 \item \label{t:2ii}
$\sum_{n\in\NN}\|v_{1,n}^k - p_{2,1,n}^k \|^2 < +\infty $\quad \text{and}\quad
$\sum_{n\in\NN}\|v_{2,n}^k - p_{2,2,n}^k \|^2 < +\infty $.
\item \label{t:2iii}
 $x_{1,n}^i\weakly \overline{x}_{1,i}$, 
$v_{1,n}^k\weakly \overline{v}_{1,k}$,
$v_{2,n}^k\weakly \overline{v}_{2,k}$ and 
 such that 
$ M_{k}^*\overline{v}_{2,k} 
=  N_{k}^*\overline{v}_{1,k}$ and that 
$(\overline{x}_{1,1},\ldots,\overline{x}_{1,m}$ solves \eqref{primal2} and 
$(N_{1}^*\overline{v}_{1,1},\ldots,N_{s}^*\overline{v}_{1,s})$ solves \eqref{dual2}.
\item 
\label{t:3iv} 
Suppose that $f_j$ is uniformly convex at $\overline{x}_{1,j}$, for some 
$j\in\{1,\ldots,  m\}$, then $x_{1,n}^j \to \overline{x}_{1,j}$. 
\item 
\label{t:3v} 
Suppose that $\varphi$ is uniformly convex at 
$(\overline{x}_{1,1},\ldots,\overline{x}_{1,m})$, then $(\forall i\in\{1,\ldots,m\})\; 
x_{1,n}^i\to \overline{x}_{1,i}$.
 \item 
\label{t:3vv} 
Suppose that $\ell^{*}_j$ is uniformly convex at $\overline{v}_{1,j}$, for some 
$j\in\{1,\ldots,  k\}$, then $v_{1,n}^j \to \overline{v}_{1,j}$. 
\item 
\label{t:3viv} 
Suppose that $g^{*}_j$ is uniformly convex at $\overline{v}_{2,j}$, for some 
$j\in\{1,\ldots,  k\}$, then $v_{2,n}^j \to \overline{v}_{2,j}$. 
\end{enumerate}
\end{corollary}
\begin{proof}  Set
\begin{equation}
\label{e:set1}
 \begin{cases}
 (\forall i\in\{1,\ldots,m\})\quad A_i = \partial f_i\quad \text{and}\quad C_i = \nabla_i\varphi,\\
 (\forall k\in\{1,\ldots,s\})\quad B_k = \partial g_k,\quad D_k = \partial\ell_k.\\
 \end{cases}
\end{equation}
Then it follows from \cite[Theorem 20.40]{livre1} that $(A_i)_{1\leq i\leq m}$, $(B_k)_{1\leq k\leq s}$,
and $(D_k)_{1\leq k\leq s}$ are maximally monotone. Moreover,
 $(C_1,\ldots,C_m) = \nabla\varphi$ is $\nu_0$-Lipschitzian. 
Therefore, every conditions on the operators in Problem \ref{prob} are satisfied.
The condition \eqref{e:dion1} 
 implies that $ (\forall k\in\{1,\ldots, s\})\; \dom (\ell_k\circ N_k)^* \not= \emp$
and $\dom (g_k\circ M_k)^* \not= \emp$. Therefore,
using \cite[Proposition 13.11]{livre1}, we have 
$ (\forall k\in\{1,\ldots, s\})\;(\ell_k\circ N_k)^*\in\Gamma_0(\GG_k)$ and
 $(g_k\circ M_k)^*\in\Gamma_0(\GG_k)$,
 so are $\ell_k\circ N_k$ and $g_k\circ M_k$.
We next derive from \eqref{e:dion1}, \cite[Proposition 12.34]{livre1}, 
\cite[Proposition 24.27]{livre1}, \cite[Propositions 16.5(ii)]{livre1} that  
\begin{alignat}{2}
\label{e:set2}
 \big(\forall k\in\{1,\ldots,s\}\big)\quad 
(M^{*}_k\circ(\partial g_k)\circ M_k)\;\vuo\;(N^{*}_k\circ(\partial \ell_k)\circ N_k) 
&\subset
\partial ((\ell_k\circ N_k)\;\vuo\; (g_k\circ M_k)). 
\end{alignat}
Let $\HHH$ and $\GGG$ be defined as in the proof of Theorem \ref{t:2}, and let
  $\LL,\MM,\NNN, \zz$ and $\rr$ be defined as in \eqref{e:acl}, and define 
\begin{equation}
\begin{cases}
 f\colon\HHH\to \left]-\infty,+\infty\right[\colon \xx\mapsto \sum_{i=1}^m f_i(x_i)\\
g\colon\GGG\to \left]-\infty,+\infty\right[\colon \vv\mapsto \sum_{k=1}^s g_k(v_k)\\
\ell\colon\GGG\to \left]-\infty,+\infty\right[\colon \vv\mapsto \sum_{k=1}^s \ell_k(v_k).
\end{cases}
\end{equation}
 Observe that \cite[Proposition 13.27]{livre1},
\begin{equation}
f^*\colon \yy\mapsto  \sum_{i=1}^m f_{i}^*(y_i),\quad
 g^*\colon \vv \mapsto \sum_{k=1}^sg_{k}^*(v_k),\; \quad \text{and}\quad  
 \ell^*\colon \vv \mapsto\sum_{k=1}^s \ell_{k}^*(v_k).
\end{equation}
We also have 
\begin{equation}
 (\ell\circ \NNN)\;\vuo\; (g\circ \MM )
\colon \vv\mapsto \sum_{k=1}^s \big((\ell_k\circ N_k)\;\vuo\; (g_k\circ M_k)\big)(v_k). 
\end{equation}
Then 
the primal problem becomes 
\begin{equation} 
\label{primal2r}
\underset{\xx\in\HHH}{\text{minimize}}\;
f(\xx) - \scal{\xx}{\zz}
+ ((\ell\circ \NNN)\;\vuo\; (g\circ \MM))(\LL\xx -\rr)
+ \varphi(\xx),
\end{equation}
and the dual problem becomes 
\begin{equation}
\label{dual2r} 
\underset{\vv\in\GGG}{\text{minimize}}\;
(\varphi^{*}\;\vuo\;  f^*)(\zz-\LL^*\vv)
+(\ell\circ \NNN)^*(\vv)  + (g\circ\MM)^*(\vv)+ \scal{\vv}{\rr}.
\end{equation}
Furthermore, the condition \eqref{e:dion2} implies that 
the set of solutions to \eqref{dualin} is non-empty. Let $(\overline{\xx},\overline{\vv}) 
= (\overline{x}_1,\ldots,\overline{x}_m,\overline{v}_1,\ldots,\overline{v}_s)$ 
be a solution to \eqref{dualin}. Then, removing $\overline{v}_1,\ldots,\overline{v}_s$ from \eqref{dualin},
we obtain,
\begin{alignat}{2}
\quad z_i\in \sum_{k=1}^s L_{k,i}^*\bigg(
\Big( (N^{*}_k\circ(\partial \ell_{k})\circ N_k)\;\vuo\; (M^{*}_k\circ(\partial g_{k})\circ M_k)
\Big)\bigg(\sum_{j=1}^m L_{k,j}\overline{x}_j-r_k\bigg)\bigg) \notag\\
\hfill+ \partial f_i(\overline{x}_i)+
\nabla_{i}\varphi(\overline{x}_1,\ldots,\overline{x}_m). 
\end{alignat}
Then, using \eqref{e:set1}, \eqref{e:set2}, 
\cite[Corollary 16.38(iii)]{livre1}, \cite[Proposition 16.8]{livre1}, 
\begin{equation}
\label{set3}
 {\boldsymbol 0}\in \partial\big( f + \scal{\cdot}{\zz}\big)(\overline{\xx})
+\LL^*\Big(\partial((\ell\circ\NNN) \;\vuo\; (g\circ \MM))(\LL\overline{\xx}-\rr)\Big)
+ \nabla \varphi(\overline{\xx}).
\end{equation}
Therefore, by \cite[Proposition 16.5(ii)]{livre1}, we derive from \eqref{set3} that
\begin{equation}
\label{set4}
 {\boldsymbol 0}\in \partial\Big( f + \scal{\cdot}{\zz} 
+ \big((\ell\circ\NNN) \;\vuo\; (g\circ \MM)\big)(\LL\cdot-\rr)
+ \varphi\Big)(\overline{\xx}).
\end{equation}
Hence, by Fermat's rule \cite[Theorem 16.2]{livre1} that 
$\overline{\xx}$ is a solution to \eqref{primal2r}, 
i.e, $\overline{\xx}$ is a solution to \eqref{primal2}.
We next remove  $\overline{x}_1,\ldots,\overline{x}_m$ from \eqref{dualin} and 
 using \cite[Theorem 15.3]{livre1} 
and \cite[Theorem 16.5(ii)]{livre1}, we obtain,
\begin{alignat}{2}
\label{set5}
 -\rr&\in -\LL\Big((\partial f
+ \nabla \varphi)^{-1}(\zz-\LL^*\overline{\vv})\Big) 
+ (\MM^{*}\circ(\partial g)\circ \MM)^{-1}\overline{\vv} + 
(\NNN^*\circ(\partial \ell)\circ \NNN)^{-1}\overline{\vv} \notag\\
& \subset -\LL\Big( \partial (f+ \varphi)^{*}(\zz-\LL^*\overline{\vv})\Big) 
+ (\partial(g\circ\MM))^{-1}\overline{\vv} + (\partial(\ell\circ\NNN))^{-1}\overline{\vv} 
\notag\\
& = -\LL\Big( \partial ( f^*\;\vuo\; \varphi^{*})(\zz-\LL^*\overline{\vv})\Big) 
+ \partial(g\circ\MM)^{*}\overline{\vv} + \partial(\ell\circ\NNN)^{*}\overline{\vv} 
\end{alignat}
Therefore, by \cite[Proposition 16.5(ii)]{livre1}, we derive from \eqref{set5} that
\begin{equation}
 {\boldsymbol 0} \in \partial \Big(  (\varphi^{*}\;\vuo\;  f^*)(\zz-\LL^*\cdot)
+(g\circ\MM)^{*}+(\ell\circ\NNN)^*
\Big)\overline{\vv}.
\end{equation}
Hence, by Fermat's rule \cite[Theorem 16.2]{livre1} 
that $\overline{\vv}$ is a solution to \eqref{dual2r}, 
i.e, $\overline{\vv}$ is a solution to \eqref{dual2}.
Now,  algorithm \eqref{e:forwardback1} is a 
special case of the algorithm \eqref{e:forwardback}. Moreover,
every specific conditions in Theorem \ref{t:2} are satisfied. 
Hence, the conclusions follow from Theorem \ref{t:2} and the fact  that 
the uniform convexity of a function in $\Gamma_0(\HH)$ at a point in its domain 
implies the  uniform monotonicity of its subdifferential at that point.
\end{proof}
\begin{remark} Here are some remarks.
\begin{enumerate}
\item In the special case when $m=1$ and $(\forall k\in\{1,\ldots,s\})\;\GG_k=\HH_1, L_{k,i} =\Id$,
algorithm \eqref{e:forwardback1} reduces to  \cite[Eq.(4.20)]{plc13b}.
\item Some sufficient conditions which ensure that \eqref{e:dion2} is satisfied
are in \cite[Proposition 4.2]{Botb}.
\end{enumerate}
 \end{remark}

The next example will be an application to the problem of recovery an ideal image from 
multi-observation \cite[Eq.(3.4)]{jmma2}. 
\begin{example} 
Let $p,K$,$(q_i)_{1\leq i\leq p}$ be a strictly positive integers, 
let $\HH = \RR^{K}$, and for every $i\in\{1,\ldots,p\}$, let
$\GG_i=\RR^{q_i}$ and $T_i\colon\HH\to\GG_i$
be a linear mapping. Consider the problem of recovery an ideal 
image $\overline{x}$ from 
\begin{equation}
 (\forall i\in\{1,\ldots,p\})\; r_i = T_i\overline{x} + w_i,
\end{equation}
 where each $w_i$ is a noise component.
Let $(\alpha,\beta,\gamma) \in \left]0,+\infty\right[^3$, 
$(\omega_i)_{1\leq i\leq p} \in\left]0,+\infty\right[^p$,
 let $C$ be a non empty, closed convex subset of $\HH$, 
model the prior information of the ideal image. Base on \cite[Eq.(5.4)]{plc13b},
we propose the following variational problem to recover $\overline{x}$, 
\begin{equation}
\label{app1}
 \underset{x\in C}{\text{minimize}} \sum_{k=1}^p\frac{\omega_k}{2}\|r_k-  T_{k}x\|^2 +   
 \gamma\|Wx\|_1+(\alpha \|\cdot\|_{1,2}\circ \nabla)\;\vuo\;(\beta \|\cdot\|_{1,2}\circ \nabla^2)(x),
\end{equation}
where $\nabla$ and $\nabla^2$ are respectively the first and the second order discrete
gradient, $W$ is an analysis operator such as wavelet operator, frame operator, 
the norm $\|\cdot\|_{1,2}$ is defined as in \cite[Eq.(5.5)]{plc13b}.
The problem \eqref{app1} is a special case of the primal problem \eqref{primal2} with
\begin{equation}
 \begin{cases}
s=2, m=1, L_{1,1}= L_{2,1} =\Id,\\
M_1 =\nabla, g_1 = \alpha \|\cdot\|_{1,2}, N_1 =\nabla^2, \ell_1 = \beta \|\cdot\|_{1,2},\\
 M_2 = W, g_2 = \gamma\|\cdot\|_1, \ell_2 =\iota_{\{0\}}, N_2 = \Id,\\
 f_1 = \iota_{C},\varphi = \sum_{k=1}^p\frac{\omega_k}{2}\|r_k-  T_{k}\cdot\|^2.
 \end{cases}
\end{equation}
Using the same argument as in \cite[Section 5.3]{plc13b}, we can check that 
\eqref{e:dion1} and \eqref{e:dion2} are satisfied. A numerical result for the 
case when $p=1$ is presented in \cite[Section 5.4]{plc13b}.
\end{example}
{\bf Acknowledgements} This research work is funded by Vietnam National Foundation for 
Science and Technology Development (NAFOSTED) under Grant No. 102.01-2012.15.

\end{document}